\documentclass[11pt]{article}

\usepackage[a4paper,margin=1in]{geometry}
\usepackage[T1]{fontenc}
\usepackage{lmodern}
\usepackage[english]{babel}

\usepackage{amsmath,amssymb,amsthm,mathtools}
\numberwithin{equation}{section}
\theoremstyle{plain}
\newtheorem{theorem}{Theorem}
\newtheorem{proposition}[theorem]{Proposition}
\newtheorem{corollary}[theorem]{Corollary}
\theoremstyle{definition}

\theoremstyle{remark}

\usepackage{algorithm}
\usepackage{algpseudocode} 

\usepackage{booktabs,tabularx,array,calc}
\usepackage{siunitx}
\usepackage{multirow}
\usepackage{graphicx,caption}
\usepackage{epstopdf} 

\usepackage{xcolor}
\usepackage{listings}
\lstdefinelanguage{Matlab}{
  keywords={break,case,catch,classdef,continue,else,elseif,end,for,function,global,if,otherwise,persistent,return,spmd,switch,try,while,parfor},
  keywordstyle=\color{blue}\bfseries,
  comment=[l]\%, commentstyle=\color{gray}\ttfamily,
  stringstyle=\color{red}, morestring=[m]', morestring=[m]"
}
\usepackage{hyperref}
\usepackage[nameinlink,capitalise]{cleveref}
\hypersetup{
  colorlinks=true,
  linkcolor=blue!50!black,
  citecolor=blue!50!black,
  urlcolor=blue!50!black,
  pdfauthor={Awad H. Al-Mohy},
  pdftitle={Computing Linear Combinations of varphi-Function Actions for Exponential Integrators}
}

\def\expm{\t{expm}}
\def\phifunm{\t{phi\_funm}}

\def\phimv{\t{phimv}}
\def\bamphi{\t{bamphi}}
\def\kiops{\t{kiops}}
\def\phipm{\t{phipm}}

\def\vphi{\varphi}

\def\exact{\mathop{\mathrm{exact}}}

\def\C{\mathbb{C}}

\def\and{\mathop{\mathrm{and}}}

\def\Htmd{\Ht@@'_{2m+1}}

\def\t#1{\texttt{\upshape #1}}

\def\a{\alpha}

\def\D{\Delta}

\def\resp{respectively}

\def\tol{\mathrm{tol}}

\def\trace{\mathrm{trace}}

\def\diag{\mathrm{diag}}

\def\1i{\mathrm{i}}
\def\e{\mathrm{e}}         

\def\norm#1{\|#1\|}
\def\normt#1{\|#1\|_2}

\def\nbyn{n \times n}

\def\R{\mathbb{R}}

\def\N{\mathbb{N}}

\def\cD{\mathcal{D}}

\title{Computing Linear Combinations of $\varphi$-Function Actions for Exponential Integrators}
\author{Awad H.~Al-Mohy}
\date{\today} 

\begin{document}
\maketitle


\begin{abstract}
We propose a matrix-free algorithm for evaluating linear combinations of $\varphi$-function actions, $w_i := \sum_{j=0}^{p} \alpha_i^{\,j}\,\varphi_j(t_i A)v_j$ for $i=1\colon r$, arising in exponential integrators.
The method combines the scaling and recovering method with a truncated Taylor series, choosing a spectral shift and a scaling parameter by minimizing a power-based objective of the shifted operator.
Accuracy is user-controlled and ultimately limited by the working precision.
The algorithm decouples the stage abscissae $t_i$ from the polynomial weights $\alpha_i^j$, and a block variant enables simultaneous evaluation of $\{w_i\}_{i=1}^r$.
Across standard benchmarks, including stiff and highly nonnormal matrices, the algorithm attains near-machine accuracy (IEEE double precision in our tests) for small step sizes and maintains reliable accuracy for larger steps where several existing Krylov-based algorithms deteriorate, providing a favorable balance of reliability and computational cost.
\end{abstract}

\section{Introduction}
The evaluation of the action of the matrix exponential and its related functions
plays a significant role in the exponential time integration methods for semilinear problem of the form
\begin{equation}\label{ivp}
u'(t) = f(t,u(t)) = A\,u(t)+g(t,u(t)), \quad u(t_0) = u_0,
\end{equation}
where $u(t)\in\C^n$, $A\in\C^{n\times n}$ is a matrix representing spatial discitization of a partial differential equation, and $g$ is a nonlinear vector function. 
Exponential integrator schemes for the problem \eqref{ivp} involve linear combinations of the form 
\cite{alhi11, grt18, hoos10, hls98, koos13, niwr12}
\begin{equation}\label{phi.comb}
w_i := \sum_{j=0}^{p} \alpha_i^{\,j}\,\varphi_j(t_i A)v_j,\quad i=1\colon r,
\end{equation}
where $t_i$ are stage abscissae, $\a_i^j$ are polynomial weights, and
\begin{equation}\label{def.phik}
	\varphi_0(tA) = \e^{tA}, \quad 
	\varphi_j(tA) = 
	\sum_{k=0}^{\infty} \frac{t^kA^k}{(k+j)!}, \quad j \in\N^+,
\end{equation}
For instance, the fourth-order six-stage exponential Runge–Kutta integrator (\textsc{expRK4s6}) proposed by Luan~\cite{valu21} takes the form:
\begin{equation}\label{expRK4s6}
\begin{aligned}
U_{n2} &= u_n + \varphi_{1}(c_{2}hA)\, c_{2}h\,f(t_n,u_n),\\
U_{n,k} &= u_n + \varphi_{1}(c_{k}hA)\, c_{k}h\,f(t_n,u_n)
          + \varphi_{2}(c_{k}hA)\,\frac{c_k^{2}}{c_{2}}\, h\, D_{n2}, && k=3,4,\\
U_{n,j} &= u_n + \varphi_{1}(c_{j}hA)\, c_{j}h\,f(t_n,u_n)
          + \varphi_{2}(c_{j}hA)\,\frac{c_j^{2}}{\,c_{3}-c_{4}\,}\, h\!\left(-\frac{c_{4}}{c_{3}}D_{n3}+\frac{c_{3}}{c_{4}}D_{n4}\right) \\
        &\quad\; + \varphi_{3}(c_{j}hA)\,\frac{2c_j^{3}}{\,c_{3}-c_{4}\,}\, h\!\left(\frac{1}{c_{3}}D_{n3}-\frac{1}{c_{4}}D_{n4}\right), && j=5,6,\\
u_{n+1} &= u_n + \varphi_{1}(hA)\, h\,f(t_n,u_n)
          + \varphi_{2}(hA)\,\frac{1}{\,c_{5}-c_{6}\,}\, h\!\left(-\frac{c_{6}}{c_{5}}D_{n5}+\frac{c_{5}}{c_{6}}D_{n6}\right) \\
        &\quad\; + \varphi_{3}(hA)\,\frac{2}{\,c_{5}-c_{6}\,}\, h\!\left(\frac{1}{c_{5}}D_{n5}-\frac{1}{c_{6}}D_{n6}\right),
\end{aligned}
\end{equation}
where $h$ is the time step size, $D_{ni} := g(t_n + c_i h, U_{ni}) - g(t_n, u_n)$ for $2 \le i \le 6$, and
$c_{2}=c_{3}=1/2$, $c_{4}= c_6 = 1/3$, $c_{5}=5/6$.
Given an algorithm for evaluating such linear combinations, scheme~\eqref{expRK4s6} requires six invocations per step (e.g., the algorithm of Niesen and Wright~\cite{niwr12}). 
However, the stages $\{U_{n3},U_{n4}\}$, on the one hand, and $\{U_{n5},U_{n6}\}$, on the other, are independent—so each pair can be computed simultaneously (in parallel) once the required increments $D_{n2}$ and $\{D_{n3},D_{n4}\}$ are available. 
Luan \emph{et al.}~\cite{lpr19}, Gaudreault \emph{et al.}~\cite{grt18}, and Caliari \emph{et al.}~\cite{ccz23} propose algorithms that reduce the number of calls per step to \emph{four} by evaluating the independent stage pairs simultaneously.

The main scope of this manuscript is to develop an algorithm that computes the linear combinations \eqref{phi.comb} simultaneously. 

Our work is organized as follows.
Section~\ref{sec.MD} derives the method, building on Al-Mohy’s analysis~\cite{almo24}, where an operator $\cD_f$ acting on a matrix function $f$ is defined and its properties are analyzed; this operator underpins our analysis and derivation.
Section~\ref{sec.param} presents a Taylor series error analysis and our strategy for selecting the algorithmic parameters: we specify how to compute the scaling parameter and an optimal spectral shift, and we introduce a truncation criterion for the Taylor series.
Section~\ref{sec.alg} describes two versions of the algorithm.
Section~\ref{sec.num} reports extensive numerical experiments.
Finally, Section~\ref{sec.conc} offers concluding remarks.

\section{Method derivation}\label{sec.MD}
In this section we present the theoretical framework of our algorithm for computing the linear combinations \eqref{phi.comb}.
We begin with the following theorem.
\begin{theorem}\label{thm1}
Let $A\in\C^{\nbyn}$, $V=\begin{bmatrix}
                       v_0 & v_p & v_{p-1} & \cdots & v_1 
                     \end{bmatrix}\in\C^{n\times (p+1)}$, $J=[0]\oplus J_p(0)$, where $J_p(0)\in\C^{p\times p}$ is 
the Jordan block associated with zero, and 
\begin{equation}\label{rec.D}
S = \sum_{k=1}^{\infty}\frac{D_k}{k!},\quad 
D_k=(A/s)D_{k-1}+D_1(J/s)^{k-1},\quad D_1=V/s,\quad s\in\N^+.
\end{equation}
Then the recurrence 
\begin{equation}\label{rec.F}
F_{k+1} = \e^{A/s}F_k + S\e^{kJ/s}, \quad F_1=S,\quad k=1\colon s-1
\end{equation}
yields
\begin{equation}\label{eq.Fk}
F_s = \begin{bmatrix}
        \vphi_1(A)v_0 & \vphi_1(A)v_p & \vphi_1(A)v_{p-1}+\vphi_2(A)v_p & \cdots & \sum_{k=1}^{p}\vphi_k(A)v_k 
      \end{bmatrix},
\end{equation}
where the $j$th column of $F_s$ is $F_s(:,j) = \sum_{k=1}^{j-1}\vphi_k(A)v_{p-j+k+1}$, $2\le j\le p+1$. 
Consequently, we have $\e^Av_0 = AF_s(:,1)+v_0$.
\end{theorem}
\begin{proof}

We have $e^z = \left(\e^{z/s}\right)^s$ for any $z\in\C$ and any natural number $s$.
Applying the linear operator $\cD$ defined in \cite[Def.~1.1]{almo24} at the ordered pair $(A,J)$ in direction $V$ using the chain rule~ \cite[Lem.~1.2\,(6)]{almo24}, we obtain
\begin{equation}\label{Dexp}
\cD_{\exp}(A, J, V) = \cD_{z^s} \left( \e^{A/s},\e^{J/s}, S\right),
\end{equation}
where $S =\cD_{\exp}(A/s,J/s, V/s)$. We have
$
\e^z = \sum_{k=0}^{\infty}z^k/k!
$.
Thus,
$
S = \sum_{k=1}^{\infty} \cD_{z^k}\left( A/s,J/s,V/s \right)/k!.
$
Since $z^k = z \cdot z^{k-1}$, we have
\[
D_k:=\cD_{z^k}\left( A/s,J/s,V/s \right) = \left( \frac{A}{s} \right)\cD_{z^{k-1}} + \left( \frac{V}{s} \right) \left( \frac{J}{s} \right)^{k-1},
\]
obtained from the product rule of the operator $\cD$. Thus, the product rule applied again to the right hand side of \eqref{Dexp} yields 
\[
\cD_{\exp}(A, J, V) 
= \e^{A/s} \cD_{z^{s-1}} + S\left(\e^{J/s}\right)^{s-1}=:F_s
\]

It remains to clarify the relation \eqref{eq.Fk}. By \cite[Eq.~(1.2)]{almo24}, $\cD_{\exp}(A,J,V)$ is the $(1,2)$ block of 
$\exp\left(\begin{bmatrix}
   A & V \\
   0 & J 
 \end{bmatrix}\right)$. Reblocking the matrix as
\[
 \begin{bmatrix}
   A & V \\
   0 & J 
 \end{bmatrix}
 =\begin{bmatrix}
    \begin{bmatrix}
   A & v_0 \\
   0 & 0
 \end{bmatrix} & \begin{matrix}
                   v_p & v_{p-1} & \cdots & v_1 \\
                   0 & 0 & \cdots & 0 
                 \end{matrix} \\
    0 & J_p(0) 
  \end{bmatrix},
\]
taking the exponential of both sides, and applying \cite[Thm.~2.1]{alhi11} on the right hand side for the matrix 
$
\begin{bmatrix}
  A & v_0 \\
  0 & 0 
\end{bmatrix}
$
and the vectors $\begin{bmatrix}
                   v_j \\
                   0 
                 \end{bmatrix}$, $j=1\colon p$ using the relation
\[\vphi_{\ell}\left(\begin{bmatrix}
    A & v_0 \\
    0 & 0 
  \end{bmatrix}\right)=\begin{bmatrix}
                         \vphi_{\ell}(A) & \vphi_{\ell+1}(A)v_0 \\
                         0 & \frac{1}{\ell!}
                       \end{bmatrix}, 
\]                                       
                 the result follows immediately.
\end{proof}
This theorem leads to the following important corollary.
\begin{corollary}
Under the assumptions of Theorem \ref{thm1} with the matrix $J$ replaced by 
$\a\, J$ for some nonzero $\a\in\C$, the first and last columns of the resulting matrix $F_s^{(\a)}$ are
\begin{equation*}
F_s^{(\a)}(:,1) = \vphi_1(A)v_0,\quad F_s^{(\a)}(:,p+1) = \vphi_1(A)v_1+\a\vphi_2(A)v_2+\a^2\vphi_3(A)v_3+\cdots+\a^{p-1}\vphi_p(A)v_p.
\end{equation*}
\end{corollary}
\begin{proof}
Consider the diagonal matrix $\D=\diag(1,\,\a,\,\a^2,\cdots,\,\a^p)$. We have 
$\a\,J=\D^{-1}J\D$. By Al-Mohy \cite[Lem.~1.2~(2)]{almo24}, we have
$
\cD_{\exp}(A,\a\,J,V) = \cD_{\exp}(A,J,V\D^{-1})\D.
$
Applying Theorem \ref{thm1} for $A$, $J$, and $V\D^{-1}=\begin{bmatrix}
                       v_0 & \a^{-1}v_p & \a^{-2}v_{p-1} & \cdots & \a^{-p}v_1 
                     \end{bmatrix}$ yields
\[
\cD_{\exp}(A,J,V\D^{-1}) = \begin{bmatrix} 
\vphi_1(A)v_0 & \a^{-1}\vphi_1(A)v_p  & \cdots & \sum_{k=1}^{p}\a^{k-p-1}\vphi_k(A)v_k
 \end{bmatrix}.
\]            
The result follows by multiplying through from the right by $\D$ and extracting the first and the last columns. 
\end{proof}

\section{Error analysis and parameter selection}\label{sec.param}
Theorem \ref{thm1} forms the basis for our algorithm. We use a truncated Taylor series, $T_m(z)=\sum_{k=0}^{m}z^k/k!$, to approximate the exponential function $\e^z$. As Taylor polynomial yields good approximation near the origin, we need to take the scaling parameter $s$ large enough to accelerate the convergence. The matrix $S$ in \eqref{rec.D} is computed once then passed to the recurrence \eqref{rec.F} to recover the effect of the scaling. We then extract the first and last columns of the matrix $F_s$ to evaluate the linear combination  \eqref{phi.comb}. If $s$ is large, the computation is dominate by the $s-1$ actions of the matrix exponential $\e^{A/s}$ on $F_k$ in \eqref{rec.F}. Therefore, it is important to keep $s$ as small as possible. Al-Mohy and Higham \cite{alhi11} propose shifting the diagonal entries of the matrix $A$ by $\trace(A)/n$, scaling the shifted matrix, premature termination of Taylor polynomial if a relative forward error achieves a given tolerance. Since the matrix $A$ is accessible via matrix-vector product only, the evaluation of $\trace(A)$ is not straightforward. 
Therefore we develop a new approach to shift $A$. We have
\begin{equation}\label{fre.bound}
  \left\|\e^{A/s}v-T_{m-1}(A/s)v\right\|\le\frac{\norm{A^{m}v}}{s^{m}m!}
  +\left\|\sum_{k=m+1}^{\infty}\frac{A^kv}{s^kk!}\right\|,
\end{equation} 
where $v$ is a unit vector. Now, consider the scalar function 
\begin{equation}\label{objve}
  \nu(\xi)=\norm{(A-\xi I)^{m}u}=\left\|\sum_{k=0}^{m}\binom{m}{k}\xi^{m-k}w_k\right\|,
\end{equation}  
where $w_{k+1}=Aw_k$ with $w_0=v$. Suppose $\xi^*$ is a minimizer of $\nu(\xi)$.
Thus, for a given tolerance, $\tol$, we choose $s$ such that $s^{-m}\nu(\xi^*)/m!\le\tol$. We use the spectral norm for its smoothness.
\begin{algorithm}
\caption{Selection of scaling and shift parameters ($s$, $\xi^\ast$)}
\label{alg:param}
\begin{algorithmic}[1]
\State \textbf{Input:} Matrix $A\in\C^{n\times n}$ or linear operator $A:\C^n\!\to\!\C^n$, degree $m\!\ge\!1$ (default $m=61$), and tolerance $\tol$ (default $\tol = 2^{-53}$).
\State \textbf{Output:} Scaling parameter $s>0$ and optimal shift $\xi^*\in\C$.
\State $g_{\max} = \delta\log(\mathrm{realmax})$, 
       $g_{\min} =  \delta\log(\mathrm{realmin})$.
       \Comment{(Over/under)flow guards (default $\delta = 0.8$)}
\State Take $v\in\R^n$ with $\|v\|_2=1$, set $V_0:=v$.
\State $\ell_j=\log\binom{m}{j}$ and $\ell_{\max}=\max_j \ell_j$,\quad $j=0,\dots,m$ 

\For{$k=1,\dots,m$} \Comment{Build safe powers}
  \State $W_k=A V_{k-1}$
  \State $L_k=\log\|W_k\|_2$
  \If{$L_k+\tfrac12\log(k+1)+\ell_{\max} > g_{\max}$ \textbf{ or } $L_k < g_{\min}$}
     \State \textbf{break}
  \Else \State $V_k = W_k$
  \EndIf
\EndFor
\State Let $r$ be the last accepted index ($V_0,\dots,V_r$ available).

\State $j_r=\max\{2,\,r-5\}$.
\State $s_0 \;=\; \exp\!\Big(\frac{1}{r-j_r}\sum_{j=j_r}^{r-1}(L_{j+1}-L_j)\Big)$
\Comment{Geometric mean via safe ratios}

\State $V_j \gets s_0^{-j}\,V_j$,\quad $j=1,\dots,r$
\Comment{Normalize and extend to degree $m$}
\For{$k=r+1,\dots,m$}
\State $V_k \gets s_0^{-1}\,A V_{k-1}$
\EndFor
\State Form $V =[V_0~V_1~\cdots~V_m]\in\C^{n\times(m+1)}$.
\State For $\xi\in\R$, let $z(\xi):=-\xi/s_0$,\quad $p(\xi)=[z^m,\dots,z,1]^T$, and $b=[\binom{m}{0},\dots,\binom{m}{m}]^T$.
\State $f(\xi)\;=\;\big\|\,V\,(b\odot p(\xi))\,\big\|_2^{1/m}$, where $\odot$ is Hadamard product.
\Comment{Objective function}
\State      $\mathcal{I}=\big[-\sqrt{n}\,s_0,\; \sqrt{n}\,s_0\big]$
      \Comment{Optimization interval}

\State $\xi^*\;=\;\arg\min_{\xi\in\mathcal{I}} f(\xi)$ 
\Comment{Brent's method}
\State $s \;=\; s_0 \;\dfrac{f(\xi^\ast)}{\big(\mathrm{tol}\cdot m!\big)^{1/m}}$
\end{algorithmic}
\end{algorithm}

Algorithm \ref{alg:param} computes the scaling parameter $s$ and the optimal shift $\xi^\ast$.
Starting from a unit vector, we build a short “power basis” $v,Av,A^2v,\ldots$ only as long as it is numerically safe, using a conservative guard that accounts both for overflow/underflow and for the later \emph{sum} of up to $k{+}1$ weighted columns in the objective (hence an extra logarithmic cushion that grows with $k$). When the guard would be violated, we stop; from the last few accepted growth factors $\|A^jv\|/\|A^{j-1}v\|$—deliberately excluding the newest column—we form a preliminary scale $s_0$ via their geometric mean. We then normalize each built column by $s_0^j$ and extend the basis to degree $m$ by repeatedly applying $A$ and dividing by $s_0$, which keeps all columns in a comfortable magnitude range. With this normalized basis, we define a one–dimensional objective that measures the size of a binomially weighted polynomial combination of the columns; the weights depend on a shift parameter only through the reduced variable $z=-\xi/s_0$. A scalar line search over a fixed bracket proportional to $\sqrt{n}\,s_0$ yields the minimizing shift $\xi^\ast$ and the corresponding objective value.
Finally, the scaling used downstream is set to
\[
s \;=\; \frac{s_0\,f(\xi^\ast)}{\big(\mathrm{tol}\,m!\big)^{1/m}},
\]
which enforces the target tolerance while inheriting stability from the guarded basis construction and the $s_0$ normalization.

Therefore, we will work on the matrix $A-\xi^* I$ with scaling parameter $s$. But how to recover $F_s=\cD_{\exp}(A,J,V)$? We have $\e^z=\e^{\xi^*}\e^{z-\xi^*}$. Applying the operator $\cD$ at $(A,J)$ in the direction $V$ yields
\begin{equation*}
\cD_{\exp}(A,J,V)= \e^{\xi^*}\cD_{\exp}(A-\xi^* I,J-\xi^* I,V).
\end{equation*} 
However, forming $\e^{\xi^\ast}$ directly may overflow or underflow, thereby destabilizing the right hand side computation.
To avoid this problem, we apply the operator $\cD$ on the equation $\e^z=\left(\e^{s^{-1}\xi^*}\e^{s^{-1}(z-\xi^*)}\right)^s$ and obtain
\begin{eqnarray}
 \cD_{\exp}(A,J,V) &=& \cD_{z^s}\left(\e^{s^{-1}\xi^*}\e^{s^{-1}(A-\xi^*I)},\e^{s^{-1}J},S\right)\nonumber \\
   &=& \e^{s^{-1}\xi^*}\left(\e^{s^{-1}(A-\xi^*I)}\cD_{z^{s-1}}\right)+ S\left(\e^{s^{-1}J}\right)^{s-1},\label{cD.shift}
\end{eqnarray}
where $S =\cD_{\exp}((A-\xi^*I)/s,(J-\xi^*I)/s,\e^{s^{-1}\xi^*}V/s)$.
The recurrence relation \eqref{rec.F} becomes
\begin{equation}\label{rec.F.shift}
F_{k+1} = \e^{s^{-1}\xi^*}\left(\e^{s^{-1}(A-\xi^*I)}F_k\right)+ S\left(\e^{s^{-1}J}\right)^{k-1},\quad F_1 = S,\quad k=1\colon s-1.
\end{equation} 

Thus, after applying $\e^{s^{-1}(A-\xi^\ast I)}$ to $F_k$, the recovery of $\e^{A/s}F_k$ follows immediately by multiplying by the scalar $\e^{\xi^\ast/s}$. This device was noted and used by Al-Mohy and Higham~\cite[sect.~3.1]{alhi11}.
The shift and scaling parameters are determined for $A$; however, evaluating the combination~\eqref{phi.comb} for multiple values of $t$ is essential. To this end, in~\eqref{cD.shift} we replace $A$ and $\xi^\ast$ with $tA$ and $t\xi^\ast$, respectively; as a consequence, the scaling parameter is updated to $\lceil t\,s\rceil$.

Given a Taylor degree $m$, the objective~\eqref{objve}—based on the leading term of the forward error bound~\eqref{fre.bound}—is used to select the scaling $s$ and the shift $\xi^\ast$ for a desired accuracy. This choice does not, by itself, guarantee that the truncation error meets the target tolerance; however, the bound can always be enforced by increasing $m$ while keeping $s$ fixed. Consequently, in \eqref{rec.F.shift} (with $A$ denoting the scaled and shifted operator), we allow the Taylor series used to approximate $\e^{A}F_k$ to terminate adaptively once the partial sums satisfy the following relative error test:

\begin{equation*}
\norm{T_{2j}(A)F_k - T_{2j-2}(A)F_k}\le\tol\,\norm{T_{2j}(A)F_k},\quad j\ge1.
\end{equation*}
Observe that $T_{2j}(A) - T_{2j-2}(A) = A^{2j-1}/(2j-1)!+A^{2j}/(2j)!$. Thus, it is equivalent to have
\begin{equation}\label{stopTm}
 \frac{\norm{A^{2j-1}F_k}}{(2j-1)!} + \frac{\norm{A^{2j}F_k}}{(2j)!}\le\tol\,\norm{T_{2j}(A)F_k},
\end{equation} 
which, in fact, the termination criterion proposed by Al-Mohy and Higham~\cite[Eq.~(3.15)]{alhi11}. We use the same termination criterion for the series in \eqref{rec.D}.
\section{Algorithm}\label{sec.alg}
In this section we describe the scaling and recovering algorithm for computing the linear combinations\eqref{phi.comb}. The algorithm returns a single or multiple combinations depending on the input data. Algorithm~\ref{alg.phi1} evaluates
$w=\sum_{j=0}^{p}\alpha^{j}\,\varphi_{j}(tA)v_{j}$
to a user-specified tolerance \(\tol\) using the scaling and recovering method with a truncated Taylor series. 
Given the scaling parameter $s$ and shift $\xi^\ast$ produced by Algorithm~\ref{alg:param}, it rescales the problem to $(t/s)(A-\xi^\ast I)$, records the safe undo factor \(\mu=\e^{t\xi^\ast/s}\), evaluates the matrix $S$ via the recurrence \eqref{rec.D} with stopping test \eqref{stopTm}, extracts the first and last columns of $S$, recovers $\e^{tA/s}v_0$, and then advances the scaled solution using the recurrence \eqref{rec.F.shift}, which terminates adaptively when \eqref{stopTm} is satisfied. 
\begin{algorithm}
\caption{Computation of $\sum_{j=0}^p \a^j\varphi_j(tA)v_j$}
\label{alg.phi1}
\begin{algorithmic}[1]
\algrenewcommand\algorithmicrequire{\textbf{Input:}}
\algrenewcommand\algorithmicensure{\textbf{Output:}}
\Require $t,\,\a\in\C$; matrix $A\in\C^{n\times n}$ or linear operator $A:\C^n\!\to\!\C^n$;  
$V=[v_0,v_1,\ldots,v_p]\in\C^{n\times(p+1)}$;
$J=[0]\oplus J_p(0)\in\C^{(p+1)\times(p+1)}$; 
$J=[0]\oplus J_p(0)\in\C^{(p+1)\times(p+1)}$;
tolerance $\tol$;
scaling parameter $s\in\R^{+}$ and shift $\xi\in\C$ via Algorithm \ref{alg:param}; 
\Ensure $w=\sum_{j=0}^p \a^j\varphi_j(tA)\,v_j$
\State $s \gets \lceil |\,t|s \rceil$, \quad $\mu = \e^{t\xi/s}$
\State $A_1= A - \xi I$,\quad $J\gets \a\,J$
\State $V\gets [v_0,\,v_p,\,v_{p-1},\cdots,\,v_1]$
\State $V \gets (\mu/s)V$
\State $\widetilde{J}= \e^{J/s}$, \quad $J \gets (J - t\xi I)/s$
\State $S = V$, \quad $D = V$, \quad $k = 1$, \quad $\sigma = 1$
\State $c_1 =\infty$, \quad $c_2 = \|D\|_\infty$
\While{$c_1 + c_2 > \tol\, \|S\|_\infty$}
  \State $k \gets k+1$, \quad $c_1 \gets c_2$, \quad $\sigma  \gets k\,\sigma $.
  \State $V \gets VJ$, \quad $D \gets tA_1\,D/s + V$,\quad $c_2 \gets \|D\|_\infty / \sigma $
  \State $S \gets S + D/\sigma $
  \Comment{$S \approx\cD_{\exp}(tA_1/s,J,V)$}
\EndWhile
\State $F = S(:,[1,p\!+\!1])$ \Comment{first \& last columns of $S$}
\State $F(:,1) \gets t\,A\,F(:,1) + v_0$ \Comment{recover $\e^{tA/s}v_0$}
\For{$j=1:s-1$}
    \State $E = F$;\quad $k\gets 0$;\quad $c_1 \gets \infty$;\quad $c_2 \gets \lVert F\rVert_\infty$
    \While{$c_1 + c_2 > \tol\,\lVert E\rVert_\infty$}
        \State $k \gets k+1$;\quad $c_1 \gets c_2$
        \State $F \gets tA_1\,F/(sk)$,\quad $c_2 \gets \lVert F\rVert_\infty$ 
        \State $E \gets E + F$
    \EndWhile
    \State $E \gets \mu\,E$ \Comment{undo spectral shift}
    \State $S \gets S\,\widetilde{J}$ 
    \Comment{advance block series side by $\e^{(t/s)J}$}
    \State $F \gets E + S(:,[1,p\!+\!1])$ \Comment{refresh the two tracked columns}
    \State $F(:,1) \gets E(:,1)$ 
\EndFor
\State \textbf{return} $F(:,1) = \e^{tA}v_0$,\quad 
$F(:,2)=\sum_{j=1}^p \a^{j-1}\varphi_j(tA)v_j$,\quad $w = F(:,1) + \a F(:,2)$
\end{algorithmic}
\end{algorithm}

Our MATLAB implementation provides flexible output options. 
If only $v_0$ is supplied (i.e., $V=v_0$), it returns $\e^{tA}v_0$. 
If $V=[v_1,\ldots,v_p]$ is supplied and $v_0$ is omitted, it returns $\sum_{j=1}^p \varphi_j(tA)v_j$. 
We omit these interface details from Algorithm~\ref{alg.phi1} to keep the presentation simple.

To motivate the block version of the algorithm, observe that the stages of exponential integrator schemes (e.g., \eqref{expRK4s6}) often require evaluating sequences of the form \eqref{phi.comb}
and evaluating $\vphi_k(t_i)v_k$ for different $t_i$'s and a fixed $k$. Below, we show how to achieve these goals.

Inspired by the work of Higham and Kandolf~\cite{hika17} on computing the action of trigonometric and hyperbolic matrix functions, where they extend the algorithm of Al-Mohy and Higham~\cite[Alg.~3.2]{alhi11} to a block version by replacing a scalar parameter with a diagonal matrix, we present the following proposition, which establishes a computational framework for the block algorithm.
\begin{proposition}
Let $g_i(x)=\sum_{j=0}^\infty a_{ij}x^j$ be convergent power series (for $i=0,\dots,p$), 
$A\in\C^{n\times n}$, $V=[v_0,\ldots,v_p]\in\C^{n\times(p+1)}$, and
$T=\diag(t_1,\ldots,t_r)\in\C^{r\times r}$. Choose nonzero
$\alpha_1,\ldots,\alpha_r\in\C$.
Define the $(p+1)\times r$ matrix $\Gamma=[\,\alpha_k^{\,i}\,]_{i=0:p,\ k=1:r}$ and, for each $j\ge0$,
$
D_j:=\diag(a_{0j},a_{1j},\ldots,a_{pj}).
$
Assume that for every $i$ and $k$, the spectrum $\sigma(t_kA)$ lies in the region of convergence of $g_i$.
Then
\[
G \;:=\; \sum_{j=0}^\infty A^{\,j}\,V\,D_j\,\Gamma\,T^{\,j}
\;=\;
\begin{bmatrix}
\displaystyle\sum_{i=0}^{p}\alpha_1^{\,i}\, g_i(t_1A)v_i &
\displaystyle\sum_{i=0}^{p}\alpha_2^{\,i}\, g_i(t_2A)v_i &
\cdots &
\displaystyle\sum_{i=0}^{p}\alpha_r^{\,i}\, g_i(t_rA)v_i
\end{bmatrix}\!.
\]
\end{proposition}

\begin{proof}
The $k$th column of $G$ is
\[
\sum_{j=0}^\infty A^{\,j}\,V\,(D_j\gamma_k)\,t_k^{\,j}
=\sum_{j=0}^\infty A^{\,j}\!\left(\sum_{i=0}^p \alpha_k^{\,i} a_{ij} v_i\right)t_k^{\,j}
=\sum_{i=0}^p \alpha_k^{\,i}\!\left(\sum_{j=0}^\infty a_{ij}(t_kA)^j\right)\!v_i
=\sum_{i=0}^p \alpha_k^{\,i} g_i(t_kA)v_i,
\]
where $\gamma_k$ is the $k$th column of $\Gamma$. The interchange of sums is justified by the convergence assumption on each $g_i$ at $t_kA$.
\end{proof}

\noindent\emph{Remark.} $\Gamma$ is a Vandermonde matrix built on the nodes $\a_1$, $\a_2$, \dots, $\a_r$. If $\a_1=\cdots=\a_r=\a$, then $\Gamma=[1,\a,\ldots,\a^p]^T\mathbf{1}_r^T$ and the formula reduces to the single-$\alpha$ case.

A significant application of $G$ arises when we set $g_i=\varphi_i$, which are encoded by the vector \(a^{(j)} = \bigl[ (j+i)!^{-1} \bigr]^T\), $i=0\colon p$, yielding
$ G =
\begin{bmatrix}
w_1 & w_2 & \cdots & w_r
\end{bmatrix}$
for \eqref{phi.comb}. If we want to evaluate $\vphi_k(t_i)v_k$ for different values of $t_i$ and a fixed $k$, we simply set all vectors to zero accept $v_k$ to obtain 
$ G =
\begin{bmatrix}
\vphi_k(t_1)v_k & \vphi_k(t_2)v_k & \cdots & \vphi_k(t_r)v_k
\end{bmatrix}$.
Although this approach is elegant and allows multiple linear combinations to be computed at once, the series can converge slowly when the spectrum of $A$ is broadly distributed in the complex plane. This motivates a block extension of Algorithm~\ref{alg.phi1}: we replace the scalars $t$ and $\alpha$ with diagonal matrices $T$ and $\Delta$, respectively, and modify the updates in Kronecker form to preserve the block structure. The resulting block version is Algorithm~\ref{alg:phifun}.
\begin{algorithm}
\caption{Computation of multiple linear combinations 
$w_i=\sum_{j=0}^p \a_i^k\,\varphi_j(t_iA)\,v_j$,\quad $i=1\colon r$}
\label{alg:phifun}
\begin{algorithmic}[1]
\algrenewcommand\algorithmicrequire{\textbf{Input:}}
\algrenewcommand\algorithmicensure{\textbf{Output:}}
\Require $t=(t_1,\,t_2,\cdots,\,t_r),\,\a=(\a_1,\,\a_2,\cdots,\,\a_r)\in\C^{r\times1}$; matrix $A\in\C^{n\times n}$ or linear operator $A:\C^n\!\to\!\C^n$; 
$V=[v_0,v_1,\ldots,v_p]\in\C^{n\times(p+1)}$; tolerance $\tol$;
scaling parameter $s\in\R^{+}$ and shift $\xi\in\C$ via Algorithm \ref{alg:param}.
\Ensure $w_i=\sum_{j=0}^p \a_i^k\,\varphi_j(t_iA)\,v_j$,\quad $i=1\colon r$
\State $A_1 = A - \xi I$.
\State $V\gets [v_0,\,v_p,\,v_{p-1},\cdots,\,v_1]$
\State $s \gets \lceil s\max_i|t_i| \rceil$
\State $T=\diag(t_1,\,t_2,\cdots,\,t_r)$, \quad $\mu = \e^{(\xi/s)T}$
\State $\D =\diag(\a_1,\,\a_2,\cdots,\,\a_r)$
\State $V \gets V(I_{p+1}\otimes\mathbf{1}_r^T\mu/s)$
\State $\widetilde{J}= \exp(J\otimes \D/s)$, \quad $J \gets J\otimes \D-\xi I_{p+1}\otimes T$
\State $S = V$, \quad $D = V$, \quad $k = 1$, \quad $\sigma = 1$
\State $c_1 =\infty$, \quad $c_2 = \|D\|_\infty$
\While{$c_1 + c_2 > \tol\, \|S\|_\infty$}
  \State $k \gets k+1$, \quad $c_1 \gets c_2$, \quad $\sigma  \gets k\,\sigma $.
  \State $V \gets VJ$, \quad $D \gets A_1\,D(I_{p+1}\otimes T/s) + V$,\quad $c_2 \gets \|D\|_\infty / \sigma $
  \State $S \gets S + D/\sigma $
\EndWhile
\State $F = S(:,[1:r,\,pr\!+\!1:(p+1)r])$ 
   \Comment{first \& last $n\times r$ block columns of $S$}
\State $F(:,1:r) \gets \,A\,F(:,1:r)T + v_0\mathbf{1}_r^T$ 
       \Comment{recover $\e^{t_iA/s}v_0$}
\State $T\gets I_2\otimes T/s$,\quad $\mu\gets I_2\otimes\mu$    
\For{$j=1:s-1$}
    \State $E = F$;\quad $k\gets 0$;\quad $c_1 \gets \infty$;\quad $c_2 \gets \lVert F\rVert_\infty$
    \While{$c_1 + c_2 > \tol\,\lVert E\rVert_\infty$}
        \State $k \gets k+1$;\quad $c_1 \gets c_2$
        \State $F \gets A_1\,FT/k$,\quad $c_2 \gets \lVert F\rVert_\infty$ 
        \State $E \gets E + F$
    \EndWhile
    \State $E \gets E\mu$ 
    \State $S \gets S\,\widetilde{J}$ 
    \State $F \gets E + S(:,[1:r,\,pr\!+\!1:(p+1)r])$ \Comment{refresh the two tracked blocks}
    \State $F(:,1:r) \gets E(:,1:r)$ 
\EndFor
\State \textbf{return}   $[w_1,\,w_2,\cdots,\,w_r] = F(:,1:r) + F(:,r+1:2r)\D$
\end{algorithmic}
\end{algorithm}

In the classical formulation of exponential Runge--Kutta methods, the stage abscissae $t_i$ appear simultaneously as scaling parameters in the arguments of the $\varphi$-functions and as polynomial weights outside, yielding combinations of the form
$
\sum_{k=0}^p t_i^k\,\varphi_k(t_i h A)v_k,
$
see, e.g., \cite{niwr12,miwr05,hoos10,grt18}.
In our setting, these roles are separated: the scaling parameter $t_i$ enters only inside the $\varphi$-functions, while the coefficients $\alpha_i$ act purely as external weights. 
This decoupling preserves the same algebraic structure as the standard stage representations but provides additional flexibility, since some integrator schemes may require different values for the abscissae $t_i$ and the polynomial weights $\alpha_i$. In particular, this makes it possible to accommodate stage times that do not coincide with the abscissae, or to simplify certain stages by setting $\alpha_i=1$.

\section{Numerical experiment}\label{sec.num}
In this section we present several numerical experiments to demonstrate the robustness of our algorithm. All computations were carried out in MATLAB R2022b on a PC equipped with an Intel\textsuperscript{\textregistered} Core\texttrademark\ i7-7700T CPU @ 2.90GHz and 16GB of RAM. The MATLAB code \phimv\ (\url{https://github.com/aalmohy/phimv}) implements our algorithm. 
We test it against the following MATLAB implementations:
\begin{enumerate}
  \item \phifunm: the algorithm of Al-Mohy and Liu \cite[Alg.~5.1]{alli25}, based on a scaling and recovering strategy that simultaneously computes several $\varphi$-functions for matrices of moderate dimension.
  \item \phipm: the Krylov method of Niesen and Wright \cite{niwr12} for evaluating a single linear combination of the form~\eqref{phi.comb} with $\alpha_i=t_i$.
  \item \bamphi: the method of Caliari, Cassini, and \v{Z}ivkovi\'c \cite{ccz23}, using Newton-form polynomial interpolation at special nodes combined with Krylov techniques (\url{https://github.com/francozivcovich/bamphi}).
  \item \kiops: the adaptive Krylov method of Gaudreault, Rainwater, and Tokman \cite{grt18}, which employs the incomplete orthogonalization procedure (\url{https://gitlab.com/stephane.gaudreault/kiops}).
\end{enumerate}
Both \bamphi\ and \kiops\ accept multiple right-hand sides and evaluate the combination in~\eqref{phi.comb} with $\alpha_i=t_i$ in a single call.

\begin{figure}
  \centering
  \includegraphics[width = 11cm]{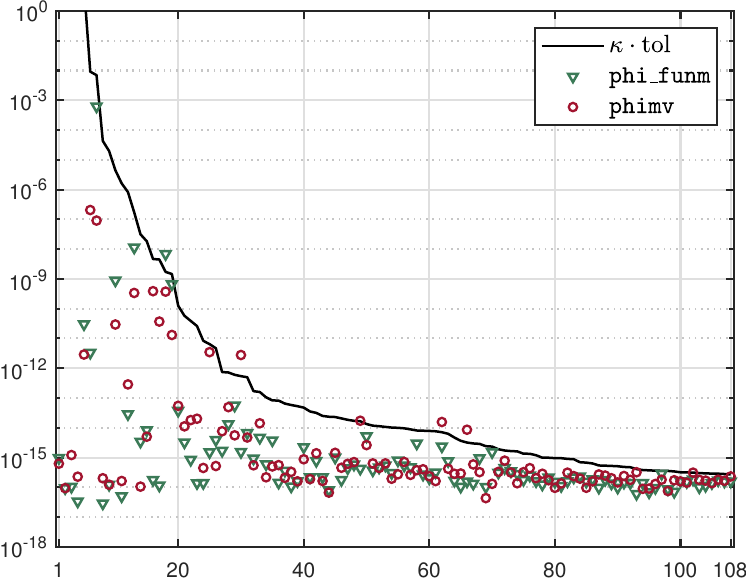}
  \caption{Relative forward errors for the computation $w=\sum_{j=0}^p\varphi_j(A)v_j$ using \phimv\ 
  and \phifunm. The solid
  line represents the condition number, $\kappa$, multiplied by $\tol=2^{-53}$.}
  \label{fig.test1}
\end{figure}
%
\begin{figure}
  \centering
  \includegraphics[width = 11cm]{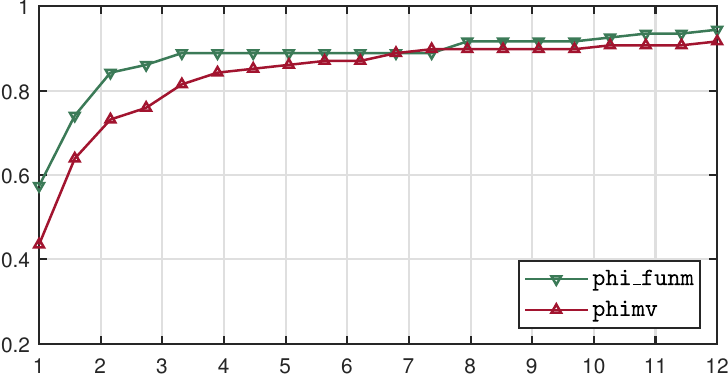}
  \caption{The data in Figure \ref{fig.test1} presented as a performance profile. }
  \label{fig.test2}
\end{figure}
\paragraph{First experiment.}
We compare \phimv\ against \phifunm\ on the suite of 108 test matrices
described and used by Al-Mohy and Liu \cite[Sec.~6]{alli25} in their stability study of \phifunm. This benchmark assesses the robustness of \phimv\ and the effectiveness of its scaling-parameter selection and termination criterion. In preliminary runs we also included \bamphi\ and \kiops, but observed frequent breakdowns and prohibitive runtimes; we therefore restrict this experiment to \phimv\ and \phifunm.

For each test matrix we computed the linear combination
$w=\sum_{j=0}^p \varphi_j(A)v_j$
using \phimv\ with inputs $t=1$, $\alpha=1$, tolerance $\tol=2^{-53}$ (double precision), and randomly generated
$V=[v_0,v_1,\ldots,v_p]$ with $p=5$.
With \phifunm, we first evaluate the matrices $\varphi_j(A)$ in a single call and then form $w$.
For each method, let $\widehat{w}$ denote its computed result in double precision.
The reference solution $w_{\exact}$ is obtained from \cite[Eq.~(2.10)]{alhi11},
\[
w=\begin{bmatrix} I_n & 0 \end{bmatrix}
\exp\!\left(
\begin{bmatrix}
A & [\,v_p,\,v_{p-1},\cdots,\,v_1\,]\\[2pt]
0 & J_p(0)
\end{bmatrix}
\right)
\begin{bmatrix} v_0\\ e_p \end{bmatrix},
\]
computed by applying the algorithm of Fasi and Higham \cite{fahi19} to the augmented matrix with 200-digit precision.

Figure~\ref{fig.test1} reports the relative forward errors $\normt{w_{\exact}-\widehat{w}}/\normt{w_{\exact}}$ of the two methods. The problems are ordered by decreasing condition number $\kappa$ (estimated as in Al-Mohy \cite[Sec.~5]{almo24}), and the black solid line shows the reference bound $\kappa\,\tol$. The errors mostly lie below this line, indicating satisfactory numerical stability for both implementations. To discriminate more finely, Figure~\ref{fig.test2} shows a performance profile \cite{domo02}: for each problem, we identify the smallest error attained by any \emph{solver} (here, \phimv\ or \phifunm); for a factor $f\!\ge\!1$, a solver counts as successful on that problem if its error is within $f$ of the best. The profile plots, versus $f$, the fraction of problems on which the solver meets this criterion. The value at $f=1$ is the win rate; as $f$ increases, the curve reflects overall robustness. Higher curves indicate better accuracy across the suite.

\paragraph{Second experiment.} We consider the spectral collocation discretization of the Laplacian with homogeneous Dirichlet boundary conditions on the interval $[0,L]$. The resulting operator $A \in \mathbb{R}^{(N-1)\times(N-1)}$ is obtained by forming the Chebyshev differentiation matrix $D$, rescaling it to the interval $[0,L]$, squaring to approximate the second derivative, and removing the first and last rows and columns to enforce the boundary conditions. The matrix $A$ is dense and highly nonnormal, with eigenvalues lying on the negative real axis and a spectral radius that grows rapidly with $N$, which makes it a stiff and representative test problem for exponential integrators. For $N=100$ and $L=2$, the following MATLAB script \cite{tref00} constructs $A$:
\begin{lstlisting}[language=MATLAB,caption={},label={code:testmatrix}]
j = (0:N)'; x = (cos(pi*j/N)+1)*(L/2); 
c = [2;ones(N-1,1);2].*(-1).^j;
X = repmat(x,1,N+1); dX = X - X';
D = (c*(1./c)')./(dX+eye(N+1)); D = D - diag(sum(D,2));
A = (2/L)^2 * (D^2); A = A(2:N,2:N);
\end{lstlisting}

\begin{table}
\centering
\setlength{\tabcolsep}{-2.0pt} 
\caption{Relative errors and run times (s) over varying $t$.}\label{tab:err.time1}
\begin{tabular*}{0.80\textwidth}{lSSSS}
\toprule
$~~t$ & $\phimv$ & $\bamphi$ & $\kiops$ & $\phipm$ \\
\midrule
\addlinespace[4pt]
\multirow{2}{*}{$10^{-4}$} & 3.5e-14 & 3.4e-13 & 1.3e-15 & 1.7e-10 \\
  & 0.013 & 0.047 & 0.014 & 0.005 \\
\addlinespace[4pt]
\multirow{2}{*}{$10^{-3}$} & 5.0e-14 & 1.2e-13 & 3.0e-14 & 3.4e-11 \\
  & 0.049 & 0.206 & 0.053 & 0.012 \\
\addlinespace[4pt]
\multirow{2}{*}{$10^{-2}$} & 5.7e-13 & 7.0e-13 & 1.5e-13 & 7.4e-06 \\
  & 0.436 & 1.546 & 0.375 & 0.012 \\
\addlinespace[4pt]
\multirow{2}{*}{$10^{-1}$} & 1.2e-12 & 1.0e-11 & 1.0e-05 & 4.2e+19 \\
  & 4.086 & 15.026 & 3.133 & 0.045 \\
\addlinespace[4pt]
\multirow{2}{*}{1} & 1.5e-10 & 4.2e-11 & 5.8e+08 & 4.2e+64 \\
  & 40.7 & 152.2 & 37.3 & 0.110 \\
\bottomrule
\end{tabular*}
\end{table}
For $p=6$, and $t_i= 10^{i-5}$, $i=1\colon5$, we computed the linear combination $w_i=\sum_{j=0}^{p}t_i^j\varphi_j(t_iA)v_j$ for a random set of vectors $v_j$'s.
Since $A$ has a relatively small size, we used the algorithm of Fasi and Higham \cite{fahi19} for reference solutions as described in the previous experiment. Table~\ref{tab:err.time1} reports the relative forward errors in 1-norm alongside the execution time for each algorithm.
All routines achieve near–machine precision for $t\le 10^{-3}$ 
with \phipm\ the fastest but already the least accurate among the four. At $t=10^{-2}$, \kiops\ remains most accurate and slightly faster than \phimv\  while \phipm\ loses several digits. For larger steps the contrast is stark: at $t=10^{-1}$, \phimv\ delivers the best accuracy at moderate cost, \bamphi\ is accurate but slower, \kiops\ degrades, and \phipm\ fails catastrophically. At $t=1$, only \phimv\ and \bamphi\ remain reliable: \bamphi\ is slightly more accurate but much slower than \phimv\, whereas \kiops\ and \phipm\ are unusable. Overall, \phimv\ offers the best robustness–efficiency trade-off across step sizes: it matches or nearly matches the most accurate algorithm for $t\le 10^{-2}$ and is the only method (apart from \bamphi, which attains similar accuracy albeit at higher cost) that maintains trustworthy accuracy for $t\in\{10^{-1},1\}$.
 
\begin{table}[t]
\centering
\setlength{\tabcolsep}{1.5pt} 
\caption{$A=UW^T$ with \eqref{M1}. Relative errors and run times (s) over varying $t$.}
\label{tab:imgeig-compact}
\begin{tabular}{l S S S}
\toprule
$t$         & {$\phimv$} & {$\bamphi$} & {$\kiops$} \\
\midrule
$0.1$   & 1.65e-16   & 1.32e-16    & 1.86e-16 \\
            & 3.08       & 1.18        & 0.55  \\
\addlinespace[2pt]
$1$   & 5.52e-15   & 1.18e-15    & 3.04e-16 \\
            & 1.36       & 0.82        & 0.50  \\
\addlinespace[2pt]
$10$   & 7.99e-13   & 1.78e-14    & 1.19e-13 \\
            & 3.37       & 0.87        & 0.67  \\
\addlinespace[2pt]
$50$   & 8.52e-13   & 3.52e-13    & 4.66e-12 \\
            & 11.56       & 0.90        & 0.91  \\
\addlinespace[2pt]
$100$   & 5.10e-12   & 1.21e-12    & 1.87e-11 \\
            & 21.91       & 0.93        & 1.36  \\
\bottomrule
\end{tabular}
\end{table}
\begin{table}[t]
\centering
\setlength{\tabcolsep}{1.0pt} 
\caption{$A=UW^T$ with \eqref{M2}. Relative errors and run times (s) over varying $t$. \textsc{NA} indicates no answer; \textsc{TO} indicates a timeout (10 min).}
\label{tab:nonnorm-compact}
\begin{tabular}{l S S S}
\toprule
$t$         & {$\phimv$} & {$\bamphi$} & {$\kiops$} \\
\midrule
$0.1$   & 9.38e-12   & 4.72e-12    & 5.60e-11 \\
            & 1.92       & 2.32        & 0.94  \\
\addlinespace[2pt]
$1$   & 1.46e-09   & 3.27e-11    & 1.10e-08 \\
            & 2.79       & 5.33        & 1.21  \\
\addlinespace[2pt]
$10$   & 3.78e-10   & 1.02e-11    & 2.43e-07 \\
            & 5.43       & 21.64        & 8.12  \\
\addlinespace[2pt]
$50$   & 1.35e-09   & 8.45e-06    & \textsc{NA} \\
            & 18.36       & 84.44        & \textsc{TO}  \\
\addlinespace[2pt]
$100$   & 1.24e-09   & 4.17e-06    & \textsc{NA} \\
            & 30.17       & 177.83        & \textsc{TO}  \\
\bottomrule
\end{tabular}
\end{table}
\begin{table}[t]
\centering
\setlength{\tabcolsep}{6pt} 
\caption{$A=UW^T$ with \eqref{M3}. Relative errors and run times (s) over varying $t$. \textsc{BD}, \textsc{NA}, and \textsc{TO} indicate breakdown, no answer, and a timeout (10 min), \resp.}
\label{tab:moler-compact}
\begin{tabular}{l S S S}
\toprule
$t$         & {$\phimv$} & {$\bamphi$} & {$\kiops$} \\
\midrule
$10^{-5}$   & 2.39e-10   & \textsc{BD}    & 1.16e-10 \\
            & 1.84       & 190.83        & 1.22  \\
\addlinespace[2pt]
$10^{-3}$   & 1.91e-09   & \textsc{NA}    & 2.32e-09 \\
            & 2.16       & \textsc{TO}       & 1.38  \\
\addlinespace[2pt]
$10^{-1}$   & 2.11e-05   & \textsc{BD}    & 1.32e-05 \\
            & 2.55       & 3.65        & 1.40  \\
\addlinespace[2pt]
$1$   & 2.22e-05          & \textsc{BD}   & \textsc{NA} \\
            & 5.25        & 2.26        & \textsc{TO}  \\
\addlinespace[2pt]
$10$   & 4.60e-05         & \textsc{BD}    & \textsc{NA} \\
            & 33.22       & 4.05        & \textsc{TO}  \\
\bottomrule
\end{tabular}
\end{table}
\paragraph{Third experiment.}
We design a reproducible family of large, low-rank test operators using the deterministic factorization
$A=U W^{\!T}$ built from a fixed orthonormal basis and a user-chosen core
$M\in\R^{r\times r}$. The goals are (i) to stress algorithms across \emph{normal} and \emph{highly nonnormal} regimes while keeping matrix–vector products matrix-free and inexpensive, and (ii) to provide a trustworthy \emph{reference solution} for accuracy assessment at very large dimensions.

We employ an orthonormal discrete-cosine basis (DCT-II)~\cite{gilb99}.
Let $U\in\R^{n\times r}$ contain the first $r$ columns of the orthonormally scaled DCT-II matrix with entries
\[
U_{ik}
=\sqrt{\frac{2}{n}}\;\alpha_k\,
\cos\!\Big(\frac{\pi\,(i+\tfrac12)\,k}{n}\Big),
\qquad i=0,\dots,n-1,\;\; k=0,\dots,r-1,
\]
where $\alpha_0=1/\sqrt{2}$ and $\alpha_k=1$ for $k\ge 1$.
We then set $W:=U M^{\!T}$. The spectrum and normality of $A$ are controlled entirely by $M$.
All applications of $A$ are matrix-free:
$x\mapsto U\,(W^{\!T}x)$ with $O(nr)$ work and storage.

For $V=[\,v_0\;v_1\;\cdots\;v_p\,]$ and $j\ge1$, one checks that
$A^j = U M^{\,j-1} W^{\!T}$. Separating the $j=0$ term yields
\begin{equation}\label{lin.comb.low.rank}
\sum_{j=0}^{p} \varphi_j(tA)\,v_j
\;=\;
\sum_{j=0}^{p} \frac{1}{j!}\,v_j
\;+\;
U\!\left(t\sum_{j=0}^{p} \varphi_{j+1}(tM)\,(W^{\!T} v_j)\right).
\end{equation}
We use the right-hand side of~\eqref{lin.comb.low.rank} as the reference; the small matrices
$\vphi_i(tM)$ are computed by the MATLAB code \phifunm\ of Al-Mohy and Liu~\cite[Alg.~5.1]{alli25}.

We consider three cores:
\begin{eqnarray}
\label{M1}
  M_1 &=& \begin{bmatrix}
          0 & 10 \\
          -10 & 0
        \end{bmatrix},\quad n=2\times10^5,\quad p=3, \\
\label{M2}
  M_2 &=& \begin{bmatrix}
          -1 & 10^5 \\
           0 & -10 \end{bmatrix},\quad n=4\times10^5,\quad p = 4,\\
\label{M3}
  M_3 &=& \begin{bmatrix}
          0 & e & 0 \\
          -(a+b) & -d & a \\
          c & 0 & -c
        \end{bmatrix},\quad n= 5\times10^5,\quad p = 2,
\end{eqnarray}
with $a=2\times10^{10}$, $b=4\times10^8/6$, $c=200/3$, $d=3$, and $e=10^{-8}$.
The matrix $M_3$ was used by Moler~\cite{Moler2012} when testing \expm\ alongside
the algorithm of Al-Mohy and Higham~\cite[Alg.~5.1]{alhi09a}, which underlies \expm\ since release R2015b.

For each $M_i$, we construct the operator $x\mapsto Ax$, randomly generate $V$, and
evaluate $\sum_{j=0}^{p}\varphi_j(tA)v_j$ at several $t$ using \phimv\ (with $\alpha=1$),
and using \bamphi\ and \kiops\ with each $v_j$ replaced by $v_j/t^j$.
We set $\tol=2^{-53}$ for \kiops\ to match the defaults of the other routines.
If a routine fails internally we record a breakdown (\textsc{BD}); if it
does not return an answer we record \textsc{NA}; and if the run exceeds 10 minutes
we terminate it and record a timeout (\textsc{TO}).
We report relative $1$-norm errors (vs.\ \eqref{lin.comb.low.rank}) and run times (s).

Table~\ref{tab:imgeig-compact} shows the results for the normal operator associated with the
skew-symmetric $M_1$. All routines achieve near machine accuracy for small $t$ and remain very
accurate as $t$ increases. \kiops\ is the fastest here, with \bamphi\ close behind and relatively
insensitive to $t$. The cost of \phimv\ increases with $t$, consistent with the growth of the spectral radius of $tA$.

Table~\ref{tab:nonnorm-compact} highlights the effect of nonnormality of the operator associated with the core $M_2$.
\phimv\ maintains reliable accuracy across all $t$ with moderate runtime growth.
\bamphi\ is also accurate for $t\le 10$ but becomes more expensive and less accurate for larger $t$.
\kiops\ is fast when it succeeds but does not return for $t\ge 50$.
Overall, \phimv\ offers the most robust performance over the full $t$-range on this nonnormal test.

Table~\ref{tab:moler-compact} uses Moler’s core. The matrix $M_3$ mixes tiny and very large scales, has eigenvalues $-63.4$, $-6.2$, and $-0.11$, and a departure-from-normality measure of $2.8\times10^{10}$. Consequently, it is notoriously challenging.
 \phimv\ is the only routine to succeed for all $t$, 
with errors from $10^{-10}$ at $t=10^{-5}$ to about $10^{-5}$ for $t\in\{\,10^{-1},10\,\}$.
\bamphi\ frequently breaks down (\textsc{BD}) or times out, and \kiops\ returns only for small $t$
(up to $10^{-1}$), failing thereafter. These outcomes underline the challenge of strong nonnormality and demonstrate the robustness of our algorithm.

In summary, all three routines deliver high accuracy on benign cases, but their behaviors diverge as the step size and nonnormality increase. \phimv\ proves the most reliable overall: it returns solutions across a wide range of settings, including highly nonnormal cases;
its main drawback is higher cost as $t$ grows due to a larger scaling parameter. \bamphi\ is highly accurate and competitively fast on easy to moderately challenging problems, yet its runtime can rise sharply for large $t$ and its accuracy may deteriorate—or it may occasionally break down—under severe nonnormality. \kiops\ is often the fastest when conditions are favorable (e.g., well-behaved operators and smaller $t$), but its reliability diminishes as $t$ grows and nonnormality increases, where failures to return or timeouts become more likely. In short, \phimv\ trades a modest increase in cost on easy cases for markedly improved robustness, while \bamphi\ and \kiops\ offer speed and accuracy primarily when the problem is not highly nonnormal.

\begin{figure}[t]
  \centering
  \includegraphics[width = 11cm]{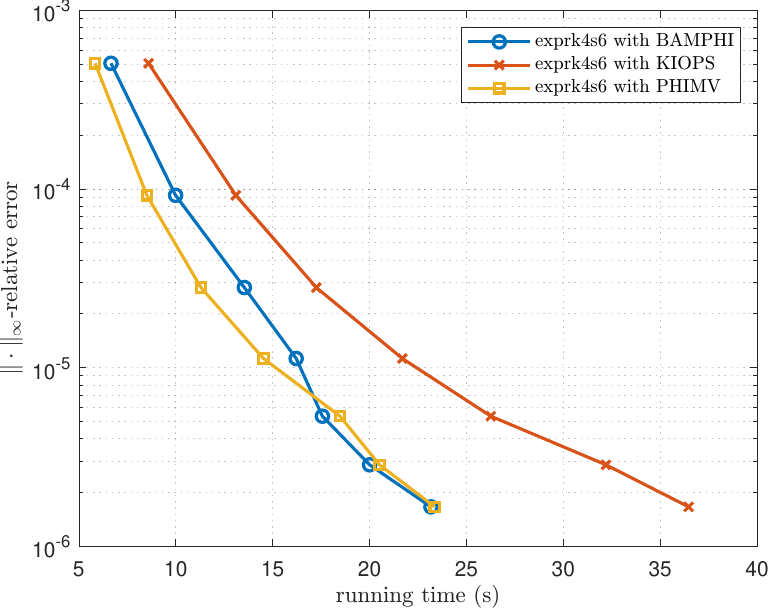}
  \caption{ADR test with \textsc{expRK4s6}: execution time (s) vs.\ $\|\cdot\|_\infty$-relative error for \bamphi, \kiops, and \phimv, all with $tol=10^{-7}$. Time steps are $h_i=2^{-8}(t_{\mathrm{end}}-t_0)/(i+1)$, $i=1,\ldots,7$. Reference solution: \texttt{ode15s} with \texttt{RelTol} $=$ \texttt{AbsTol} $=10^{-12}$.}

  \label{fig:adr}
\end{figure}
%
\paragraph{Fourth experiment: Advection–Diffusion–Reaction (ADR).}
We consider the 2D ADR~\cite{ccz23, grt18}
\begin{equation}\label{eq:adr}
  u_t
  \;=\;
  \varepsilon\,\Delta u
  \;-\; \alpha\,(u_x + u_y)
  \;+\; \gamma\,u\!\left(u-\tfrac12\right)(1-u),
\end{equation}
posed on the space–time cylinder $[t_0,t_{\mathrm{end}}]\times\Omega$ with
$\Omega=[0,1]^2$, $t_0=0$, and $t_{\mathrm{end}}=\tfrac12$. The parameters are
$\varepsilon=10^{-3}$, $\alpha=-\tfrac12$, and $\gamma=1000$. The initial condition is
\[
  u(0,x,y)=256\,x^2y^2(1-x)^2(1-y)^2.
\]
Boundary conditions are homogeneous Neumann on the edges where $xy=0$
(i.e., $x=0$ or $y=0$) and homogeneous Dirichlet on the remaining edges
($x=1$ or $y=1$):
\[
  \partial_\nu u(t,x,y)=0 \ \text{on } \{(x,y)\in\partial\Omega:\; xy=0\},\qquad
  u(t,x,y)=0 \ \text{on } \{(x,y)\in\partial\Omega:\; xy\neq 0\}.
\]

We discretize \emph{in space} on a uniform grid with $N_x=N_y=100$ points per
direction, obtaining a semi-discrete ODE of the form \eqref{ivp},
$u'(t)=A\,u(t)+g(u(t))$, where $A\in\R^{n\times n}$, $n=10^4$, collects the diffusion–advection terms and
$g$ the pointwise reaction. For time integration we use the scheme \textsc{expRK4s6} \eqref{expRK4s6} with step sizes
$h_i=2^{-8}\,(t_{\mathrm{end}}-t_0)/(i+1)$ for $i=1,\ldots,7$. We implement the routines with $\tol=10^{-7}$.
The routine \phimv\ uses the \emph{block} variant of our algorithm
(Algorithm~\ref{alg:phifun}) to evaluate parallel stages simultaneously.

Figure~\ref{fig:adr} shows the time–accuracy trade-off for this setup, plotting
executuion time (s) versus $\|\cdot\|_\infty$-relative error. The curves for \phimv\ and \bamphi\ lie close together and consistently to the left of \kiops, indicating that they reach a given accuracy in less time. At moderate accuracies, \phimv\ is marginally more efficient than \bamphi; at the tightest accuracy (about $10^{-6}$) they nearly coincide. In contrast, \kiops\ reduces the error smoothly with time but remains to the right across the range, requiring more time to match the accuracies achieved by \phimv\ and \bamphi. The reference solution is computed by the MATLAB function \texttt{ode15s} with tolerances: \texttt{RelTol} $=$ \texttt{AbsTol} $=10^{-12}$.
\section{Conclusions.}\label{sec.conc}
We presented a matrix-free algorithm for evaluating linear combinations of $\varphi$-actions,
$w_i=\sum_{j=0}^{p}\alpha_i^{\,j}\,\varphi_j(t_iA)v_j$. It combines the scaling and recovering method with a truncated Taylor series and, \emph{a priori}, determines a spectral shift and a scaling parameter by minimizing a power-based objective for the shifted operator; these two quantities are computed once and then reused across subsequent calls (e.g., across stages or time steps), avoiding per-call retuning. Accuracy is user-controlled (ultimately limited by the working precision), and an adaptive stopping test avoids unnecessary work. The computational cost is also easily predictable: for a given $t$ it is essentially proportional to the effective scaling count $\lceil |t|\,s\rceil$ (the product of the time-step magnitude and the selected scaling parameter), so the cost grows \emph{linearly} with this product. A block variant decouples the stage abscissae $t_i$ from the polynomial weights $\{\alpha_i^{\,j}\}$, enabling simultaneous evaluation and improving cache locality and Level-3 BLAS utilization when the operator acts on blocks. Extensive numerical experiments show that the algorithm attains high accuracy on benign problems and maintains reliable accuracy for larger time steps and severe nonnormality. Compared with existing Krylov-based algorithms, it offers a favorable robustness–cost trade-off on challenging instances.

\section*{Funding}
This work was supported by the Deanship of Scientific Research at King Khalid University Research Groups Program (grant RGP. 1/318/45).


\def\noopsort#1{}\def\hbk{hardback}\def\pbk{paperback}


\end{document}